\documentclass{amsart}
\allowdisplaybreaks[1]
\usepackage{mathtools} \mathtoolsset{showonlyrefs}
\usepackage{verbatim}
\usepackage{enumitem}

\newtheorem{theorem}{Theorem}
\newtheorem{lemma}[theorem]{Lemma}
\newtheorem{corollary}[theorem]{Corollary}
\newtheorem{question}[theorem]{Question}

\renewcommand {\c}{\operatorname{c}}
\newcommand {\vol}[2][n]{\operatorname{vol}_{#1}\left(#2\right)}
\newcommand {\cat}{\operatorname{cat}}

\newcommand{\R}{\mathbb R}
\newcommand{\Z}{\mathbb Z}

\newcommand{\sign}{\operatorname{sign}}

\begin{document}

\title[Centroids of sections of convex bodies]{Centroids of sections of convex bodies and Lusternik-Schnirelmann category}

\author{
  J.~Haddad,
    C.~H.~Jiménez,
    R.~Villa
}
%\author[J.~Haddad]{J.~Haddad\footnotemark[1]}
%\author[C.~H.~Jiménez]{C.~H.~Jiménez\footnotemark[1]}
%\author[R.~Villa]{R.~Villa\footnotemark[1]}
\date{\today}

\begin{abstract}
Given two symmetric convex bodies $L \subseteq K \subseteq \R^n$ with $L$ strictly convex, we prove that there exist at least $n$ hyperplanes $H$ tangent to $L$, such that the center of mass of $H \cap K$ belongs to $\partial L$.

The theorem makes use of Lusternik-Schnirelmann category theory.
\end{abstract}

\maketitle

%\renewcommand{\thefootnote}{\fnsymbol{footnote}}

%\section{TO DO list}
%\begin{enumerate}
	%\item
%\end{enumerate}
\section*{Introduction and main results}

The purpose of this short note is to draw attention to a useful tool from differential topology, that can be employed to solve some interesting problems in the domain of convex geometry, provided they have a variational structure.

The study of centroids and related geometric functionals has long provided a meeting ground between convex geometry, analysis, and topology. Problems involving the location of the centroid of sections or projections often reveal deep structural properties of convex bodies, and connect naturally with questions of symmetry, uniqueness, and stability. In many situations, these problems can be rephrased in variational form, where topological tools become effective in guaranteeing existence or multiplicity of solutions. This variational viewpoint, which treats the centroid as a critical point of an associated functional on the sphere or on a family of hyperplanes, serves as the common framework behind several classical and modern results. In what follows, we recall some of these problems and the progress that has been made toward their resolution.

Let $K\subset \R^n$ be a convex body, (a convex, compact set with non-empty interior). The centroid, center of mass or barycenter of $K$ is defined as 
\[c(K)=\frac 1{\vol{K}} \int_K x dx,\]
where $\vol{K}$ denotes de $n$-dimensional volume (Lebesgue measure) of $K$.
If $H \subseteq \R^n$ is an affine hyperplane, we will also consider the center of mass of the section $K \cap H$, defined by
\[c(K \cap H) = \frac 1{\vol[n-1]{K \cap H}} \int_{K \cap H} x dx,\]
where integration is with respect to the $(n-1)$ dimensional Lebesgue measure in $H$.

%An important number of problems concerning centroids of sections of convex bodies have been considered, let us go through some of them.
%
In \cite{grunbaum1961} and \cite{fenchelproc} Grünbaum and Löwner asked the following questions:

\begin{question}[Grünbaum]\label{quest_grun}
Is the centroid $c(K)$ of a convex body $K\subset \R^n$ the centroid of at least $n+1$ different $(n-1)$-dimensional sections of $K$ through $c(K)$? 
\end{question} 

\begin{question}[Löwner]\label{quest_low} Let $\mu(K)$ be the number of hyperplane sections of $K$ passing through $c(K)$ whose centroid is the same as $c(K)$. Let $$\mu(n)=\min_{K\in\mathcal{K}^n}\mu(K),$$ what is the value of $\mu(n)$?.

\end{question}
Let us note that these problems only make sense for non symmetric convex bodies. For any symmetric convex body $K$ (this is, if $K=-K$), both the centroid of $K$ and of any central section of $K$ are at the origin.

Question \ref{quest_grun} was answered positively by Grünbaum and Löwner for $n=2$ (see also \cite{patakova2022barycentric} for a discussion in dimension $3$). 
In a recent paper Myroshnichenko, Tatarko and Yaskin \cite{myroshnychenko2025answers} showed the existence of a convex body of revolution in $\R^n$ for $n\geq 5$ such that its centroid coincides with the centroid of exactly one section, thus answering Question \ref{quest_grun} in the negative. The question for $n=3,4$ remains open.
%Grünbaum [Gr3] claim to have a positive answer for $n\geq 3$, however, Patakova et al [patakova2022barycentric] found it to be incorrect.
In \cite{patakova2022barycentric} it was shown that for each $K\subset \R^n$ for $n\geq3$ there exists a point $p$ in the interior of $K$ which is the centroid of $4$ hyperplane sections passing through $p$.
The existence of a point in dimension $n\geq 4$ which is the centroid of $n+1$ hyperplane sections, is unknown.

Sections of convex bodies not passing through the origin were considered in many uniqueness problems.
The famous question of Barker and Larman \cite{barker2001determination} asks whether the size of the sections at a fixed distance from the origin determine the body.
More generally they ask:
\begin{question}[Barker-Larman]
	\label{quest_bar}
Let $K$ and $L$ be convex bodies in $\R^n$ containing a convex body $M$ in their interiors. Suppose that for every hyperplane $H$ that supports $M$ we have
	\[\vol[n-1]{K \cap H} = \vol[n-1]{L \cap H}.\]
Does this imply that $K=L$?
\end{question}
Partial results were obtained by the same authors in \cite{barker2001determination}, Santaló \cite{santalo1951two}, Yaskin \cite{yaskin2015thick}, Yaskin and Yaskina \cite{yaskin2011unique}, and Xiong, Ma and Cheung \cite{xiong2008determination}.
Question \ref{quest_bar} is open in general, even for $n=2$.

The connection between convex geometry and topology has been shown in some results around Questions \ref{quest_grun} and \ref{quest_low}. In the work of Patáková, Tancer, and Wagner \cite{patakova2022barycentric}, who revisited Grünbaum's question on barycentric cuts through a convex body by analyzing the \textit{depth function} on the unit sphere $S^{n-1}$. Their approach relied on \textit{differential-topological methods}, 
%treating the sphere as a smooth manifold and
interpreting barycentric hyperplanes as \textit{critical points} of a $C^1$ function. Using arguments inspired by \textit{Morse theory}, they established the existence of multiple barycentric cuts through a point of maximal depth. In contrast, earlier results by Blagojević \cite{blagojevic2016local} and Karasev \cite{karasev2011geometric} had employed \textit{algebraic-topological techniques}, using \textit{cohomological and characteristic-class methods}---notably \textit{Stiefel--Whitney classes}---to derive multiplicity results through \textit{Borsuk--Ulam--type arguments}. In this framework, one studies equivariant maps between manifolds (such as projective spaces and spheres) and applies cohomological obstructions to prove the existence of symmetric geometric configurations.

Complementing these developments, our work provides an \textit{alternative approach based on the Lusternik--Schnirelmann (LS) category}, which serves as a \textit{homotopy-invariant measure of topological complexity} and offers a conceptual parallel to classical Morse theory. 
%In our paper we use the stronger critical point theory of Lusternik-Schnirelmann, which quantifies the number of critical points of a real function, according to topological properties of its domain.
As in the approach of Patáková--Tancer--Wagner, we work within a mildly smooth setting, asking only $C^1$ functions, but employ the LS category to estimate the \textit{minimal number of critical points} of smooth functions arising in geometric constructions. Recall that the LS category $\mathrm{cat}(X)$ of a topological space $X$ is defined as the smallest number of
%open
closed
subsets that are contractible in $X$ and whose union covers $X$; it is invariant under homotopy equivalence and provides a lower bound for the number of critical points of any smooth function on $X$.
%This viewpoint yields comparable multiplicity results through an intrinsically topological argument, emphasizing the global structure of the underlying manifold rather than the specific nature of the differential data.
%aflojá con el chatgpt!

In this work we obtain a property for a pair of symmetric, convex bodies in $\R^n$ about centroids of non-central sections, in the spirit of the questions of Grünbaum and Löwner.

We say that a convex body is strictly convex if each supporting hyperplane touches the boundary at a single point.
%Equivalently, its support function is $C^1$ outside the origin.
Our main theorem is:
\begin{theorem}
	\label{thm_main}
	Let $K,L \subseteq \R^n$ be symmetric, convex bodies with $L \subseteq K$ and assume $L$ is strictly convex.
	Then there exist at least $n$ distinct pairs of supporting hyperplanes $H$ of $L$
	such that $\c(K \cap H) \in \partial L$.
\end{theorem}

We obtain the following two corollaries, by taking $L$ and $K$ to be Euclidean balls, respectively:
\begin{corollary}
	\label{cor_normal}
	Let $L \subseteq \R^n$ be a symmetric, strictly convex body.
	Then there exist $n$ distinct pairs of points $x \in \partial L$ for which the supporting hyperplane is orthogonal to $x$.

	Equivalently, there exist $n$ orthogonal projections onto $1$-dimensional subspaces $P_i:\R^n \to V_i$ for which $P_i(L) \subseteq L$.
\end{corollary}
In a similar result \cite[Proposition 2.1]{haddad2025higher}, in the non-symmetric case, at least one (not necessarily orthogonal) projection is obtained. This result also makes use of topological arguments, namely the Brouwer topological degree.

\begin{corollary}
	\label{cor_normal}
	Let $K \subseteq \R^n$ be a symmetric convex body.
	Then for every $t>0$ which is less than the inradius of $K$, there exist $n$ distinct directions $v_1, \ldots, v_n \in S^{n-1}$ for which
	\[\c(K \cap \{x: \langle x, v_i \rangle = t\}) = t v_i.\]
\end{corollary}

\section{Proof of Theorem \ref{thm_main}}

To prove Theorem \ref{thm_main}, we restate the problem in variational terms.
The idea is to define a suitable smooth functional on the sphere $S^{n-1}$ whose critical points correspond to hyperplanes tangent to $L$ for which the centroid of the section $K \cap H$ lies on $\partial L$.
This formulation transforms the geometric problem into one of locating critical points of a smooth functional on the sphere.
The remainder of the proof proceeds by analyzing this functional using Lusternik--Schnirelmann theory.
In particular, the topology of the real projective space provides the necessary lower bound on the number of critical points, ensuring the existence of the required tangent hyperplanes. We begin by fixing notation and recalling basic smoothness properties of the functions involved.

We will say that a real function defined in the Euclidean sphere $f:S^{n-1} \to \R$ is $C^1$ if it is the restriction of a real $C^1$ function defined in an open subset of $\R^n$, containing $S^{n-1}$.
The gradient of the restriction of $f$ to $S^{n-1}$ (as a smooth function on a Riemannian manifold) at a point $x \in S^{n-1}$ is
\begin{equation}
	\label{eq_def_crit}
	\nabla (\left. f \right|_{S^{n-1}})(x) = P_{x^\perp}[\nabla f(x)],
\end{equation}
where $P_{x^\perp}$ denotes the orthogonal projection onto the hyperplane orthogonal to $x$, and $\nabla f(x)$ denotes the usual gradient, as a function defined in an open set of $\R^n$.
Thus, a critical point of $\left. f \right|_{S^{n-1}}$ is a point $x \in S^{n-1}$ for which $\nabla f(x)$ is parallel to $x$.

Let $f:S^{n-1} \to \R$ be a $C^1$ function which is even. If $x$ is a critical point of $f$, then clearly $-x$ is also a critical point. We will refer to $\{x, -x\}$ as a pair of critical points.

The general theory of Lusternik-Schnirelmann category (see \cite{cornea2003lusternik}, in particular Theorem 1.15 or Theorem 21 in \cite{oprea2014applications}) states that a smooth function on a smooth manifold $M$ has at least $\cat(M)+1$ distinct critical points.
Its is known (see \cite[Example 1.8]{cornea2003lusternik} or Theorem 3 in \cite{oprea2014applications}) that the real projective $n$ space $\R P^n$ satisfies $\cat(\R P^n) = n$.

Since every real even smooth function on $S^{n-1}$ can be factored through the quotient $S^{n-1}/\Z_2 = \R P^{n-1}$, one can deduce the following theorem which is the celebrated Lusternik-Schnirelmann Theorem.
\begin{theorem}
	\label{thm_LS}
	Let $f:S^{n-1} \to \R$ be a $C^1$ even function, then $f$ has at least $n$ distinct pairs of critical points.
\end{theorem}
The precise smoothness of the function $f$ varies in the literature.
A detailed proof for $C^1$ functions can be found in the work of Rabinowitz \cite[Theorem 4.1]{rabinowitz1973some}.

Theorem \ref{thm_LS} is topological in nature, meaning that no other structure of the set of critical points can be deduced, other than its cardinality.
To see this, consider a smooth even function $f:S^{n-1} \to \R$ with a finite number $m$ of critical points $x_1, \ldots, x_m \in S^{n-1}$, and any other set of points of the same cardinality, $y_1, \ldots, y_m \in S^{n-1}$.
It is known (see for example the Homogeneity Lemma in \cite{milnor1997topology}) that one can construct a (smooth, even) diffeomorphism $\varphi:S^{n-1} \to S^{n-1}$ taking $y_i$ to $x_i$.
Then the composition $f \circ \varphi$ will have exactly $y_1, \ldots, y_m$ as critical points.

To apply Theorem \ref{thm_LS} we will consider a specific function in the following lemma.
The proof is partly contained in \cite{patakova2022barycentric}. We include the proof for completeness.

\begin{lemma}
	\label{lem_differentiability}
	Consider a convex body $K \subseteq \R^n$, and the function $V:\R_+ \times S^{n-1} \to \R_+$ given by
	\[V(t,x) = \vol{\{y \in K: \langle y, x \rangle \geq t\} }.\]
	Then $V$ is $C^1$ on the set $\{(t,x) \in \R \times S^{n-1}: 0 < V(t,x) < \vol{K}\}$.

	Moreover, denoting
	\[H_{t,x} = \{y \in \R^n : \langle x, y \rangle = t\},\]
	we have
	\[\frac \partial{\partial t} V(t, x) = \vol[n-1]{K \cap H_{t,x}} \]
	%and for any $w \perp x$,
	and
	%\[\langle \nabla_x V(t, x), w \rangle = \int_{K \cap H_{t,x}} \langle w, z \rangle dz.\]
	\[\nabla_x V(t, x) = \int_{K \cap H_{t,x}} z\, d z.\]
\end{lemma}
\begin{proof}
	Consider $(s,y) \in \R \times S^{n-1}$ close to $(t,x)$.
	Since we are assuming that $V(t,x) \in (0, \vol{K})$, we may assume that $V(s,y)$ belongs to the same interval.
	Moreover, the surfaces $H_{s,y}$ and $\partial K$ intersect transversally (meaning that $H_{s,y}$ is not a supporting plane of $K$), and we can find a point $p \in x^\perp$ in the relative interior of $P_{x^\perp}(K \cap H_{s,y})$ for every $(s,y)$ close to $(t, x)$.

	%For any set $U$ and $(s,y) \in \R \times S^{n-1}$ denote $U_{s,y} = \{z \in U: \langle z, y \rangle \geq s\}$.
	Denote now for any $(s,y) \in \R \times S^{n-1}$, $K_{s,y} = \{z \in K: \langle z, y \rangle \geq s\}$.
	Let $d$ denote the Hausdorff distance between compact subsets of $\R^n$. We clearly have for some constant $c>0$,
	\begin{enumerate}[label=\alph*), ref=\alph*]
		%\item \label{eq_dK} $d(K_{s,y}, K_{t,x}) \leq c \|x-y\|_2+c|t-s|$,
		\item \label{eq_DKH} $d(K_{t,x} \Delta K_{s,y}, K \cap H_{t,x}) \leq c \|x-y\|_2+c|t-s|$,
		\item \label{eq_H} $d(K \cap H_{t,x}, K \cap H_{s,y}) \leq c \|x-y\|_2+c|t-s|$.
	\end{enumerate}
	All the three properties are a consequence of the transversal intersection between $\partial K$ and $H_{t,x}$.
	(See Figure \ref{fig_estimates}.)
	\begin{figure}
		\includegraphics[scale=.4]{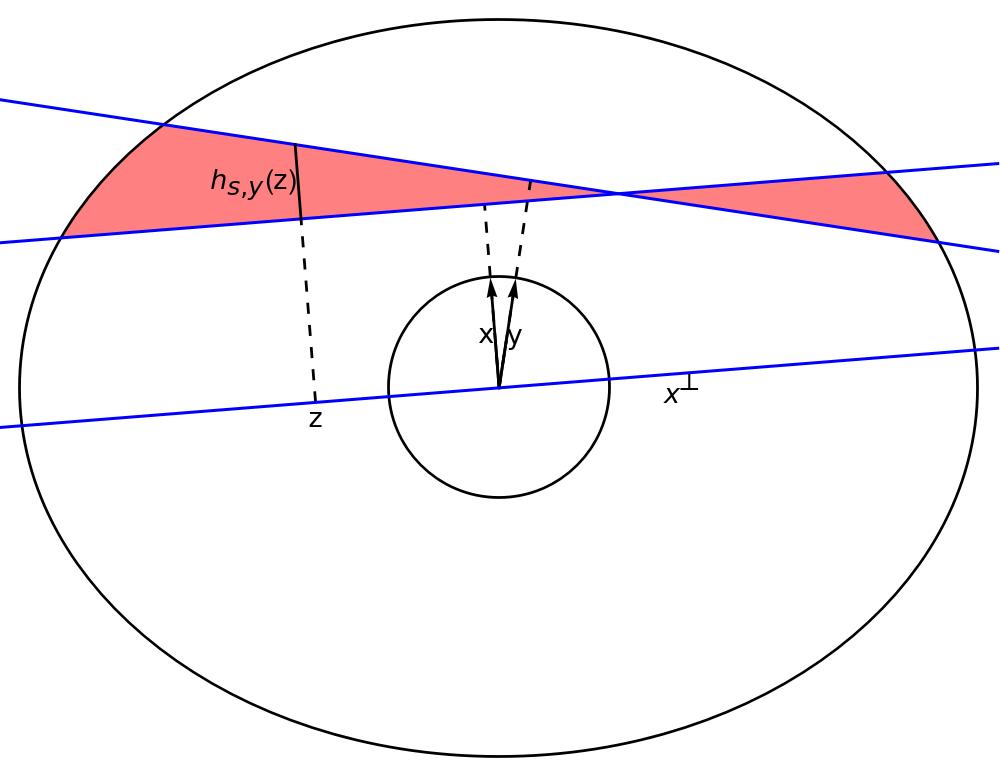}
		\caption{Sets $K_{s,y} \Delta K_{t,x}$ and $K \cap H_{t,x}, K \cap H_{s,y}$.}
		\label{fig_estimates}
	\end{figure}

	For $z \perp x$, define $h_{s,y}(z)$ by the property $z + t x - h_{s,y}(z) x \in H_{s,y}$.
	By \eqref{eq_DKH} we have 
	\begin{equation}
		\label{eq_h}
		|h_{s,y}| \leq c \|x-y\| + c |t-s|.
	\end{equation}
	Let $\rho_{s,y}$ be the radial function of $P_{x^\perp}[K \cap H_{s,y}]$ with respect to $p$, and similarly, let $\rho_{t,x}$ be the radial function of $P_{x^\perp}[K \cap H_{t,x}]$, with respect to $p$.

	For $z \in x^\perp$ denote $l(z)$ the length of the chord of $K_{t,x} \Delta K_{s,y}$ passing through $z$ and parallel to $x$.
	If $v \in S^{n-1} \cap x^\perp$ and $r \in (0, \min\{\rho_{t,x}(v), \rho_{s,y}(v)\})$ then we have both $p+r v + t x$ and $p + r v + t x - h_{s,y}(p+t v) x$ inside $K$ so the whole segment from $p+r v+tx$ to $p+r v + t x - h_{s,y}(p+t v) x$ belongs to $K$. Moreover, it belongs to $K_{t,x} \Delta K_{s,y}$.
	We deduce that $r \in (0, \min\{\rho_{t,x}(v), \rho_{s,y}(v)\})$ implies $l(p+r v) = |h_{s,y}(p+r v)|$.
	Now we can estimate
	\begin{align}
		V(s,y) - V(t,x)
		&= \vol{K_{s,y}} - \vol{K_{t,x}} \\
		&= \int_{K_{t,x} \Delta K_{s,y}} \chi_{K_{s,y}}(z) - \chi_{K_{t,x}}(z) dz\\
		&= \int_{S^{n-2}} \int_0^{\rho_{P_{x^\perp}[K_{t,x} \Delta K_{s,y}]}} r^{n-2} l(p+r v) \sign(h_{s,y}(p+r v))d r d v \\
		&= \int_{S^{n-2}} \int_0^{\min\{\rho_{t,x}(v), \rho_{s,y}(v)\}} h_{s,y}(p+r v) r^{n-2} d r d v \\
		&+ \int_{S^{n-2}} \int_{\min\{\rho_{t,x}(v), \rho_{s,y}(v)\}}^{\rho_{P_{x^\perp}[K_{t,x} \Delta K_{s,y}]}} r^{n-2} l(p+r v)\sign(h_{s,y}(p+r v)) d r dv,
	\end{align}
	where the last integral is less than $c\|x-y\|^2+c|t-s|^2$, in view of \eqref{eq_DKH} and \eqref{eq_h}.
	We get
	\begin{align}
		V(s,y) - V(t,x)
		&= \int_{S^{n-2}} \int_0^{\min\{\rho_{t,x}(v), \rho_{s,y}(v)\}} h_{s,y}(p+r v) r^{n-2} d r d v + o(\|x-y\|+|t-s|) \\
		&= \int_{S^{n-2}} \int_0^{\rho_{t,x}(v)} h_{s,y}(p+r v) r^{n-2} d r d v \\
		&+ \int_{S^{n-2}} \int_{\rho_{t,x}(v)}^{\min\{\rho_{t,x}(v), \rho_{s,y}(v)\}} h_{s,y}(p+r v) r^{n-2} d r d v + o(\|x-y\|+|t-s|) \\
	\end{align}
	Where again the last integral is less than $c \|x-y\|^2+c|t-s|^2$, in view of \eqref{eq_H} and \eqref{eq_h}. We get
	\begin{equation}
		 V(s,y) - V(t,x) = \int_{K \cap H_{t,x}} h_{s,y}(z) d z + o(\|x-y\|+|t-s|).
	\end{equation}
	Using the Taylor approximation of $h_{s,y}(z)$ at $(s,y) = (t,x)$, which is given by
	\begin{align}
		h_{s,y}(z)
		&= t - \frac{s-\langle z,y \rangle}{\langle x,y \rangle} \\
		&= - (s-t) + \langle z, y-x \rangle + \frac {\|x-y\|^2}{2 \langle x, y \rangle}(s-\langle z,y \rangle),
	\end{align}
	we get
	\[V(s,y) = V(t,x) + \int_{K \cap H_{t,x}} -(s-t) + \langle y-x, z \rangle d z + o(\|x-y\|+|t-s|),\]
	and the lemma is proven.

	%Even better, we can compute explicitly
	%\[h_{s,y}(z) = t - \frac{s-\langle z,y \rangle}{\langle x,y \rangle} \]
	%\[h_{s,y}(z) = h_{t,x}(z) - (s-t) + \langle z, y-x \rangle + \frac {2\|x-y\|^2}{\langle x, y \rangle}(s-\langle z,y \rangle)\]
	Finally, once we know that the function $V$ is differentiable, we observe that the differential $(\frac {\partial}{\partial t} V(t,x), \nabla_x V(t,x) )$ is a continuous function of $(t,x)$, thanks to property \eqref{eq_H}.
	
\end{proof}

By the characterization \eqref{eq_def_crit} we obtain immediately the following corollary:
\begin{corollary}
	\label{cor_derivative}
	Let $h:S^{n-1} \to \R$ be a $C^1$ positive function and $K$ a convex body.
	Consider the function $V_{K,h}:S^{n-1} \to \R$ defined by
	\[V_{K,h}(v) = \vol{\{x \in K: \langle x, v \rangle \geq h(v) \}}.\]
	Then the $V_{K,h}$ is $C^1$ and its gradient is given by
	\[\nabla V_{K,h}(x) = -\nabla h(x) \vol{K \cap H_{x,h(x)}} +  P_{x^\perp}[\int_{K \cap H_{x,h(x)}} y d y ]. \]
\end{corollary}

Now we are ready to prove the main theorem:
\begin{proof}[Proof Theorem \ref{thm_main}]
	Let $f:S^{n-1} \to \R$ be the function defined by
	\[f(x) = V_{K,h_L(x)}(x).\]
	Since $L$ is strictly convex, $h_L$ is a $C^1$ function (see \cite[Section 2.5]{schneider2013convex} ).
	Thus, by Corollary \ref{cor_derivative}, it is a $C^1$ function.
	By the symmetry of $K$ and $L$, it is clear that $f$ is an even function.

	From Theorem \ref{thm_LS} we deduce the existence of $n$ pairs of critical points of $f$.

	Let $z$ be one of these critical points. We will prove that the supporting hyperplane of $L$, $H = H_{z,h_L(z)}$ touches $\partial L$ at the centroid of the section.

	By Corollary \ref{cor_derivative} and the fact that $\nabla V_{K,h}(z) = 0$,
	\begin{align}
		\nabla (\left.h_L\right|_{S^{n-1}})(z)
		&= P_{z^\perp}[\frac 1{\vol[n-1]{K \cap H}} \int_{K \cap H} y dy] \\
		&= P_{z^\perp}[\c(K \cap H)].
	\end{align}
	On the other hand, since $h_L$ is homogeneous of degree $1$,
	\begin{align}
		\nabla h_L(z)
		&= h_L(z) z + \nabla (\left.h_L\right|_{S^{n-1}})(z) \\
		&= P_z[\c(K \cap H)] + P_{z^\perp}[\c(K \cap H)] \\
		&= \c(K \cap H).
	\end{align}
	It is known (see \cite[Corollary 1.7.3]{schneider2013convex} that $\nabla h_L(z)$ is the unique point in $\partial L$ which intersects the supporting hyperplane of $L$ perpendicular to $z$, which is exactly $H_{z,h_L(z)}$.
	Then $\c(K \cap H_{z,h_L(z)}) = \partial L \cap H_{z,h_L(z)}$ and the theorem is proven.
	
\end{proof}

\section*{Acknowledgments}

The authors were partially supported by Ministry of Science and Innovation, Spain, project PID2022-136320NB-I00. The first named author was supported by Grant RYC2021 - 031572 -I, funded by the Ministry of Science and Innovation / State Research Agency /10.13039 / 501100011033 and by the E.U. Next Generation EU/Recovery, Transformation and Resilience Plan.

\bibliographystyle{abbrv}
\bibliography{references}

\begin{comment}

\end{comment}

\bigskip\bigskip

\noindent\textsc{Departamento de Análisis Matemático, Universidad de Sevilla}\\
\textsc{C/ Tarfia s/n, Sevilla 41012, Spain}

\medskip

\noindent\textit{Email addresses:}\\
J.~Haddad: \texttt{jhaddad@us.es}\\
C.~H.~Jiménez: \texttt{carloshugo@us.es}\\
R.~Villa: \texttt{villa@us.es}
\end{document}